\documentclass[10pt,a4paper,reqno,oneside]{amsart}
\usepackage{amsmath,amsthm,amsfonts,amssymb}
\usepackage[usenames,dvipsnames]{color}
\usepackage[noadjust]{cite}
\usepackage{mathrsfs,mathtools}
\usepackage[extension=pdf]{hyperref}
\usepackage{todonotes}
\usepackage[top=1.65in,bottom=1.65in,left=1.5in,right=1.5in,centering]{geometry}
\definecolor{darkblue}{rgb}{0.13,0.13,0.39}
\definecolor{darkpurple}{RGB}{102,0,102}
\hypersetup{colorlinks=true,urlcolor=cyan,citecolor=darkblue,linkcolor=darkpurple,pdftitle={The stochastic heat equation, 2D Toda equations and dynamics for the multilayer process}}

\newtheorem{theorem}{Theorem}[section]
\newtheorem{lemma}[theorem]{Lemma}

\newtheorem{proposition}[theorem]{Proposition}

\theoremstyle{remark}

\newcommand{\rd}[1] {\mathrm{d}#1}

\newcommand{\W}[2]{W(\mathrm{d}#1,\mathrm{d}#2)}
\newcommand{\mW}[3]{W^{\otimes #1}(\mathrm{d}#2,\mathrm{d}#3)}

\newcommand{\bbm}{\begin{bmatrix}}
\newcommand{\ebm}{\end{bmatrix}} 

\newcommand{\mb}[1] {\mathbf{#1}}
\newcommand{\R}{\mathbf{R}}
\newcommand{\E}{\mathbf{E}}
\newcommand{\x}{\mathbf{x}}
\newcommand{\y}{\mathbf{y}}
\newcommand{\z}{\mathbf{z}}

\newcommand{\T}{\tau}

\newcommand{\V}{\Vert}
\newcommand{\p}{\prime}

\begin{document}

\title[The stochastic heat equation]{The stochastic heat equation, 2D Toda equations and dynamics for the multilayer process} 
\author{Chin Hang Lun}
\address{C. H. Lun,
  Mathematics Institute,
  University of Warwick,
  Coventry,
  CV4 7AL,
  UK}
\email{c.h.lun@warwick.ac.uk}

\author{Jon Warren}
\address{
  J. Warren,
  Department of Statistics,
  University of Warwick,
  Coventry,
  CV4 7AL,
  UK}
\email{j.warren@warwick.ac.uk} 
  
\maketitle
\thispagestyle{empty}

\begin{abstract}
We show that solutions of the stochastic heat equation driven by space-time white noise, although not smooth, meaningfully solve the two-dimensional Toda equations. Then by extending our arguments we show the time evolution of the multilayer process introduced by O'Connell and Warren, \cite{OW16}, is conjugate to a flow induced by the stochastic heat equation. In particular this establishes a Markov property conjectured by O'Connell and Warren. It also defines, for the first time, the multilayer process started from a general initial condition. 
\end{abstract}

\section{Introduction}\label{sec:intro}

In \cite{OW16}, O'Connell and Warren introduced the following: for each $n = 1,2,\ldots$, $t>0$ and $x$, $y\in\R$ define
\begin{equation}
  Z_n(t,x,y) = p_t(x-y)^n \bigg(1 + \sum_{k=1}^\infty \int_{\Delta_k(t)} \int_{\R^k} R_k(\mb{s},\mb{y^\prime}; t,x,y) \;\mW{k}{\mb{s}}{\mb{y}^\prime} \bigg),
  \label{eq:ZnChaos}
\end{equation}
where $\Delta_k(t) = \{0 < s_1 < s_2 < \cdots < s_k < t\}$. $\mb{s} = (s_1,\ldots,s_k)$, $\mb{y}^\prime = (y_1^\prime,\ldots,y_k^\prime)$ and $R_k(\mb{s}, \mb{y}^\prime ; t,x,y)$ is the $k$-point correlation function for a collection of $n$ non-intersecting Brownian bridges each of which starts at $x$ at time 0 and ends at $y$ at time $t$.
$p_t(x-y)= (2\pi t)^{-1/2} e^{-(x-y)^2/2t}$ is the transition density of Brownian motion.
The integral is a multiple stochastic integral with respect to space-time white noise. 
It was shown in \cite{OW16} by considering local times of non-intersecting Brownian bridges that the infinite sum in the definition is convergent in $L^2$ with respect to the white noise.

Observe that $u = Z_1$ is the solution to the (multiplicative) stochastic heat equation (SHE) with delta initial data:
\begin{equation}
 \begin{cases}
  \partial_t u(t,x,y) = \Big( \frac{1}{2} \Delta_y + \dot{W}(t,y) \Big) u(t,x,y), \quad t\in(0,\infty), y\in\mathbf{R}, \\
  u(0,x,y) = \delta(x-y), \quad x\in\mathbf{R}.
 \end{cases}
 \label{eq:SHEDeltaX}
\end{equation}
By a solution to the above we mean a random field $u$ which satisfies almost surely the mild form of the equation:
\begin{equation}
  u(t,x,y) = p_t(x-y) + \int_0^t \int_\R p_{t-s}(y-y^\prime) u(s,x,y^\prime) \;\W{s}{y^\prime}.
  \label{eq:SHEMild}
\end{equation}
Iterating equation (\ref{eq:SHEMild}) multiple times gives the chaos expansion (\ref{eq:ZnChaos}) for $n=1$.
One can express $Z_n(t,x,y)$ in a more suggestive notation:
\begin{equation}
  Z_n(t,x,y) = p_t(x-y)^n \mathbf{E}_{x,y;t}^X \bigg[ \mathscr{E}\mathrm{xp} \bigg( \sum_{i=1}^n \int_0^t W(s,X_s^i) \;\rd{s} \bigg) \bigg],
  \label{eq:ZnFeymannKac}
\end{equation}
where $(X_s^1,\ldots,X_s^n, 0\leq s\leq t)$ denotes the trajectories of the above mentioned collection of $n$ non-intersecting Brownian bridges and $\E_{x,y;t}^X$ is the corresponding expectation.
$\mathscr{E}\mathrm{xp}$ is the \emph{Wick exponential} defined by $\mathscr{E}\mathrm{xp}(M_t) := \exp\big( M_t - \frac{1}{2}\langle M,M\rangle_t \big)$ for a martingale $M$.
The Feynman--Kac formula (\ref{eq:ZnFeymannKac}) is not rigorous as it is unclear how one would define the integral of the white noise along a Brownian path and moreover to exponentiate such an expression.
However, one can obtain an rigorous expression by replacing $W$ in (\ref{eq:ZnFeymannKac}) with a smoothed version of the space-time white noise.
Indeed, Bertini and Cancrini showed in \cite{BC95} that such expression (in the case $n=1$) has a meaningful limit as one takes away the smoothing and that the limit solves the SHE.

Bymeans of  the Feynman--Kac formula, one can interpret the solution to the stochastic heat equation as the partition function (up to a multiplication by the heat kernel) of the continuum directed random polymer, and  then similarly, $Z_n$ is the partition function of a natural extension of the continuum polymer involving multiple Brownian paths.
The Cole--Hopf solution $h = \log u$ to the KPZ equation with narrow wedge initial data corresponds  via the Feynman--Kac formula to the free energy of the continuum directed random polymer.
With this interpretation $h$ can be regarded as the continuum analogue of the longest increasing subsequence of a random permutation, length of the first row of a random Young diagram, directed last passage percolation and free energy of a discrete/semi-discrete polymer in random media etc., see \cite{BDJ99}, \cite{BDJ99b}, \cite{BOO00}, \cite{Joh99}, \cite{Joh01}, \cite{PS02}, \cite{Joh03}, \cite{COSZ14} and the references therein.
In each of these discrete models, there is further structure provided either by multiple non-intersecting up-right paths on lattices, multi-layer growth dynamics or Young diagrams constructed from the RSK correspondence.
The work in the above mentioned references have shown that in some cases, utilisation of this additional structure have lead to derivations of exact formulae for the distribution of quantities of interest.
The above mentioned discrete models provide examples of what is called integrability or exact solvability. The motivation for introducing the partition functions $Z_n$, which are the continuum analogue of the structures mentioned above, is that they should provide insight to the integrable structure in the continuum setting. 

There has been other recent work on multiple polymer paths  and the multilayer process in the stochastic heat equation setting. In \cite{DL15} and \cite{DL16}, in a manifestation of the exact solvability, the Bethe Ansatz is used to make exact and asymptotic distributional statements. In \cite{CH13}, the KPZ line ensemble, which is expected to be given by the logarithm of the multilayer process we are considering here is shown to have a remarkable Gibbs property which generalises that of the Airy line ensemble, \cite{CH14}. Very recent results in \cite{CN16} show the multilayer process arises as a scaling limit of discrete polymer models.

It was shown in \cite[Proposition 3.3 and 3.7]{OW16} by considering a smooth space-time potential that $(Z_n, n\geq 1)$ should satisfy a system of coupled SPDEs, however unfortunately it is not immediately obvious that such SPDEs make sense in the white noise setting.
Nevertheless, it does suggests that the process should have a Markovian evolution.
Indeed, we have the following theorem which is the main result of this paper.
\begin{theorem}
  For each $n\geq 1$ and $x\in\mathbf{R}$, 
  \begin{equation}
    \big(Z_1(t,x,\cdot),\ldots,Z_n(t,x,\cdot) ; t\geq 0\big),
    \label{eq:multiLayer}
  \end{equation}
  is a Markov process with respect to the filtration generated by the space-time white noise taking values in the space ${\mathfrak L}_n:= C(\mathbf{R})\times\cdots\times C(\mathbf{R})$, where $C(\R) := C(\R,\R_+)$ is the space of continuous functions from $\R$ to $\R_+ = (0,\infty)$.
  \label{thm:Markov}
\end{theorem}

In the case of the stochastic heat equation ($n=1$), the Markov property can be seen from the Feynman--Kac formula since $u(s+t,x,y)$ can be written in the form 
\[
  \E_{x,y;t}^X\Big[ \mathscr{E}\mathrm{xp}\big( F_X(0,s)\big) \mathscr{E}\mathrm{xp}\big( F_X(s,t)\big)\Big],
\]
where $F_X(s,t)$ is a function of the Brownian bridge $X$ starting from $(s,X_s)$ and ending at $(t,y)$ and the white noise over the time interval $[s,t]$, which is independent of the white noise over $[0,s]$ and the bridge from $(0,x)$ to $(s,X_s)$.
From this one obtains the flow property of $u$:
\begin{equation}
  u(t,x,y) = \int_\R u(s,x,z) u(s,t,z,y) \;\rd{z},
  \label{eq:flowSHE}
\end{equation}
where $u(s,s+\cdot,z,\cdot)$ is the solution to the SHE driven by the shifted white noise $\dot{W}(s+\cdot,\cdot)$.
However, this argument does not apply for $n\geq 2$ since the definition of $Z_n$ involves non-intersecting Brownian bridges with common starting and ending points but at any intermediate time each of the bridges are at distinct locations.
Nevertheless, Theorem \ref{thm:Markov} is true and we shall prove it by considering a natural extension $M_n$ of $Z_n$ which  corresponds to allowing the multiple polymer paths to have differing starting and ending points from one another. This extended process can easily be seen to have the Markov property, and the key to understanding Theorem \ref{thm:Markov} is that the extension $M_n$ turns out to be able to be recovered from the values of $(Z_1, Z_2, \ldots, Z_n)$. In fact we establish in  Theorem \ref{mainthm2} what amounts to a conjugacy between the random dynamical systems  that  describe the evolution of $(M_1, M_2, \ldots, M_n)$ and $(Z_1, Z_2, \ldots, Z_n)$ as shown in  Figure \ref{figure}. This develops an idea that was suggested in \cite{OW16}, but only rigorously established when $n=2$.

\begin{figure}
\label{figure}
\includegraphics[height=4cm, trim= 12cm 21cm 11cm 4cm ]{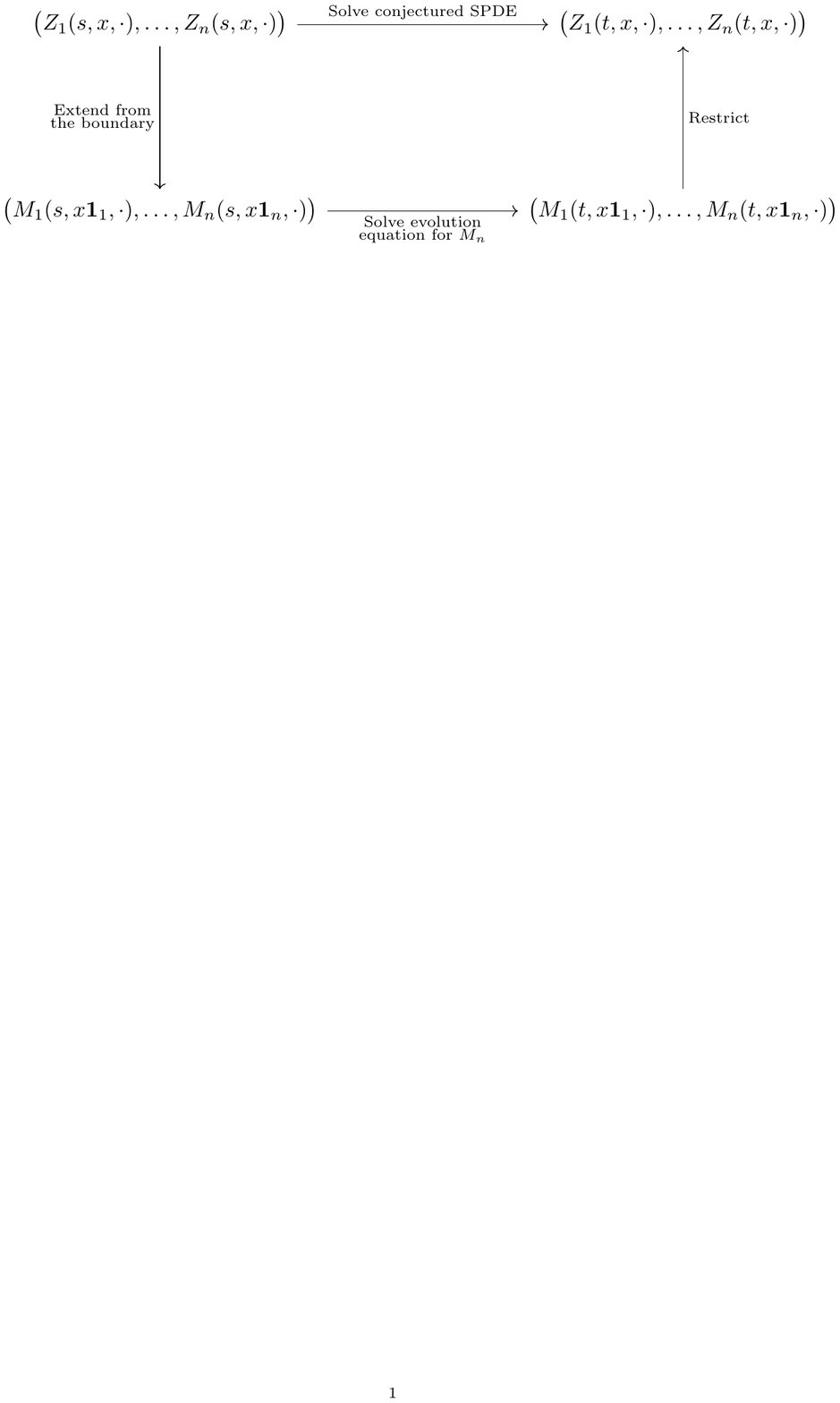}
\caption{The evolution of the multilayer process is described by a conjugacy}
\end{figure}

Let  $W_n = \{x\in\R^n : x_1\geq\cdots\geq x_n\}$ be the Weyl chamber in $\R^n$, then define for $n\geq 1$, $(t,\x,\y)\in (0,\infty)\times W_n\times W_n$,
\begin{equation}
  M_n(t,\mb{x},\mb{y}) = \frac{p_n^*(t,\mb{x},\mb{y})}{\Delta_n(\x)\Delta_n(\y)} \bigg(1 + \sum_{k=1}^\infty \int_{\Delta_{k}(t)} \int_{\R^k} R_k(\mb{s},\mb{y}^\prime; t,\mb{x},\mb{y}) \;\mW{k}{\mb{s}}{\mb{y}^\prime} \bigg),
  \label{eq:MnChaos}
\end{equation}
where $R_k$ is the $k$-point correlation function of a collection of $n$ non-intersection Brownian bridges which starts at $\mb{x}$ at time 0 and ends at $\mb{y}$ at time $t$.
$p_n^*(t,\mb{x},\mb{y}) = \det[p_t(x_i-x_j)]_{i,j=1}^n$ is by the Karlin--McGregor formula \cite{KM59} the transition density of Brownian motion killed at the boundary of $W_n$ and $\Delta_n(\mb{x}) = \prod_{1\leq i<j\leq n} (x_i-x_j)$ is the Vandermonde determinant.
Notice that $M_n(t,\x,\y)$ is well defined for $\x$, $\y$ at the boundary of the Weyl chamber since $p_n^*(t,\x,\y)/\Delta_n(\x)\Delta_n(\y)$  is a smooth function of $(\x,\y)$ over $\R^n\times\R^n$ by \cite[Lemma 5.11]{BBO09} and the $k$-point correlation function $R_k$ extends continuously to the boundary.
Moreover we can extend $M_n$ by symmetry to a function on $\R^n\times\R^n$, $M_n(t,\x,\y)$ satisfying $M_n(t, \pi \x,\sigma \y) = M_n(t,\x,\y)$ for any permutations $\pi$, $\sigma$ of $\{1,2,\ldots,n\}$. 

We have the following which is the main result of \cite{LW15}.
\begin{theorem}
\label{thm:MnRegularity}
  For all $n\geq 1$, $M_n$ has a version that is continuous over $(0,\infty)\times {\mathbf R}^n \times {\mathbf R}^n$ and $\mathbb{P}[M_n(t,\x,\y) > 0 \text{ for all } t>0 \text{ and } \x,\y\in {\mathbf R}^n] = 1$.
\end{theorem}

Moreover, when all the coordinates of $\x$ and $\y$ are equal, $M_n$ agrees almost surely up to a strictly positive multiplicative constant with $Z_n$, that is
\begin{equation}
  M_n(t,a\mb{1}_n,b\mb{1}_n) =  c_{n,t} Z_n(t,a,b), \quad c_{n,t} = c_n t^{-n(n-1)/2}, c_n = \left(\prod_{i=1}^{n-1} i!\right)^{-1},
  \label{eq:MnBoundary}
\end{equation}
where $\mb{1}_n= (1,\ldots,1) \in{\mathbf R}^n$.
This implies the almost sure continuity and everywhere strict positivity of $Z_n$ for all $n$.

It was shown in \cite{Ha13} that the difference of two solutions to the KPZ equation (with the same white noise) starting from two different H\"older continuous initial data is in $C^{\frac{3}{2}-\varepsilon}$ for every $\varepsilon>0$. 
Since $M_2(t,\x,\cdot)$ and $u(t,x,\cdot)$ are H\"older continuous of order up to 1/2, Theorem \ref{thm:MnRegularity} and equation \eqref{ratio} imply that the ratio of $u(t,x_1,\cdot)$ and $u(t,x_2,\cdot)$ is in $C^{\frac{3}{2}-\varepsilon}$ for every $\varepsilon>0$ and hence generalises the result of \cite{Ha13} to delta initial data. In fact Theorem \ref{thm:MnRegularity} implies that  there is much more regularity present when considering multiple  solutions to the stochastic heat equation driven by the same white noise than might be initially expected.  

We have the following determinantal expression for $M_n$ (see \cite[Proposition 3.2]{OW16}, and also proved in the Appendix to this paper)
\begin{equation}
  M_n(t,\x,\y) = \frac{\det[u(t,x_i,y_j)]_{i,j=1}^n} {\Delta_n{(\x)}\Delta_n{(\y)}}, \quad t>0, \x,\y\in {\mathbf R}^n_{\neq},
  \label{eq:MnKM}
\end{equation}
where the entries in the determinant are solutions to (\ref{eq:SHEDeltaX}) each driven by the same white noise.
${\mathbf R}^n_{\neq}$ denotes the subset of points in ${\mathbf R}^n$ whose coordinates are all distinct.
In view of this representation Theorem \ref{thm:MnRegularity} implies we can meaningfully assign values to Wronskians whose entries are formally given as derivatives of solutions of the stochastic heat equation. 

Suppose that we replace the space-time white noise with a smooth space-time potential in the definition of $Z_n$ then it can be shown that $Z_n$ is given by the bi-directional Wronskian
\begin{equation*}
  Z_n(t,x,y) = c_n t^{n(n-1)/2} \det[\partial_x^{i-1} \partial_y^{j-1} u(t,x,y) ]_{i,j=1}^n,
\end{equation*}
where $u(t,x,y)$ is the solution to (\ref{eq:SHEDeltaX}) driven by the smooth potential.
Now let $\tau_n(x,y) = \det[\partial_x^{i-1} \partial_y^{j-1} u(t,x,y) ]_{i,j=1}^n$ then $\tau_n$ satisfies the two-dimensional Toda equation (2DTE) 
\begin{equation*}
  \partial_{xy} q_n = e^{q_{n+1}-q_n} - e^{q_n-q_{n-1}}, \quad n\geq 1,
\end{equation*}
where $q_n = \log(\T_n/\T_{n-1})$ or equivalently,
\begin{equation}
  \partial_{xy} \log \T_n = \frac{\T_{n-1}\T_{n+1}}{\T_n^2},
  \label{eq:2DTE}
\end{equation}
with the convention that $\T_0 \equiv 1$.
Evaluating the derivative and rearranging we obtain
\begin{equation}
  \T_n \partial_{xy} \T_n - (\partial_x\T_n) (\partial_y\T_n) = \T_{n-1}\T_{n+1}.
  \label{eq:TodaBilinear}
\end{equation}
We introduce the following notation.
For an $(n+1)\times(n+1)$ determinant $D$, let $D\begin{bmatrix} i \\ j\end{bmatrix}$ be the $n\times n$ determinant obtained from $D$ by removing the $i$th row and the $j$th column and similarly let $D\begin{bmatrix} i & j \\ k & l\end{bmatrix}$ be the $(n-1)\times(n-1)$ determinant obtained from $D$ by removing the $i$th and $j$th rows and the $k$th and $l$th columns.
Then, by properties of Wronskians (\ref{eq:TodaBilinear}) can be written as
\[
  \T_{n+1}\begin{bmatrix} n+1 \\ n+1\end{bmatrix} \T_{n+1}\begin{bmatrix} n \\ n\end{bmatrix} - \T_{n+1} \begin{bmatrix} n \\ n+1\end{bmatrix} \T_{n+1} \begin{bmatrix} n+1 \\ n\end{bmatrix} = \T_{n+1} \begin{bmatrix} n & n+1 \\ n & n+1\end{bmatrix} \T_{n+1},
\]
which is nothing but the Jacobi identity for determinants \cite[equation 2.73]{Hir04}.

In Section \ref{sec:wronskians} we study Wronskians defined as the continuous extensions of determinantal expressions such as \eqref{eq:MnKM}  where $u$ is not necessarily differentiable. With the help of the Jacobi idenitity we show that such generalised Wronksians satisfy the 2D Toda equations in an integrated form. Then in Section \ref{sec:conjugacy}, using a key lemma from the preceeding section, we construct the conjugacy that describes the evolution of the multilayer process. A fuller description of this conjugacy is the content of Section \ref{sec:dynamics} which follows now.

\subsection*{Acknowledgements}

The research of C.H.L. was supported by EPSRC grant number EP/H023364/1 through the MASDOC DTC.
This research of J.W. was  supported in part by the National Science Foundation under Grant No. NSF PHY11-25915.

\section{Description of the dynamics of the multilayer process}
\label{sec:dynamics}

The multilayer process is, for a fixed $x\in {\mathbf R}$, and $n\geq 1$ the finite sequence of ``lines'' 
 $\big(Z_1(t,x,\cdot),\ldots,Z_n(t,x,\cdot) )$,  which we consider to be evolving in the time variable  $t>0$. Actually it seems more natural to work with a slightly different normalisation than that  originally used by \cite{OW16}, and consider instead the process
${\boldsymbol\tau}_{t}(x,\cdot)=\big(\tau_1(t,x,\cdot),\ldots, \tau_n(t,x,\cdot) )$,
where
\begin{equation}
\tau_n(t,x,y)=c_{n}^{-1} t ^{-n(n-1)/2} Z_n(t,x,y)= c_n^{-2} M_n( t,x{\mathbf 1}_n, y{\mathbf 1}_n).
\end{equation}

As discussed in the Introduction, we describe the evolution of the multilayer process with the aid of the extended process $M_n$. It was shown in \cite{LW15} that $M_n$, defined by the chaos expansion \eqref{eq:MnChaos}, satisfies an evolution equation which can be regarded as a type of multi-dimensional SHE.
\begin{equation}
  M_n(t,\x,\y) = \frac{p_n^*(t,\x,\y)}{\Delta_n(\x)\Delta_n(\y)} + \frac{1}{(n-1)!} \int_0^t \int_{\R^n} Q_{t-s}(\y,\y^\p) M_n(s,\x,\y^\p) \;\rd{\y_*^\p} \;\W{s}{y_1^\p},
  \label{eq:MnSPDE}
\end{equation}
almost surely for all $(t,x,y)\in(0,\infty)\times\R^n\times\R^n$ where $\rd{\y_*^\p} = \rd{y_2^\p}\ldots\rd{y_n^\p}$.
The function $Q_t(\x,\y) = \frac{\Delta_n(\y)}{\Delta_n(\x)} p_n^*(t,\x,\y)$ is the transition density of Dyson Brownian motion \cite{Dy62} starting from $\x\in W_n$ and ending at $\y\in W_n$.
Note that in (\ref{eq:MnSPDE}) we have extended $Q_t$ by symmetry to a function on $\R^n\times\R^n$ so the integral over $\R^n$ is defined.
We can consider this same  evolution  for initial data $g$, which is assumed to be a symmetric function on ${\mathbf R}^n$,
\begin{multline}
  M^g_n(t,\y) =\frac{1}{n!} \int_{{\mathbf R}^n} g(\y^\p)Q_{t}(\y,\y^\p) \;\rd{\y^\p}  + \\
\frac{1}{(n-1)!} \int_0^t \int_{\R^n} Q_{t-s}(\y,\y^\p) M^g_n(s,\y^\p) \;\rd{\y_*^\p} \;\W{s}{y_1^\p},
  \label{eq:MnSPDEg}
\end{multline}
For bounded $g$, existence and uniqueness for this equation were established in  \cite{LW15}.
If $g$ takes the form 
\begin{equation}
g(\x)= \frac{\det[f_i(x_j)]_{i,j=1}^n} {\Delta_n{(\x)}}, \quad \x\in {\mathbf  R}^n_{\neq},
\end{equation}
then we have the representation
\begin{equation}
\label{eq:REP}
  M_n^g(t,\y)=\frac{\det[u^{f_i}(t,y_j)]_{i,j=1}^n} {\Delta_n{(\y)}}, \quad t>0, \y\in  {\mathbf  R}^n_{\neq},
\end{equation}
where $u^{f_i}$ denotes the solution to the solution to the SHE with initial condition at time $0$ given by $f_i$. This is a consequence of equation \eqref{eq:MnKM}.

From now on we use the notation 
\[
  f_1 \wedge f_2 \wedge \ldots \wedge f_n (\x)= \det[f_i(x_j)]_{i,j=1}^n
\]
We let ${\mathfrak F}_n$  be the set of sequences $({\mathfrak f}_1, {\mathfrak f}_2, \ldots, {\mathfrak f}_n)$  with  each ${\mathfrak f}_k:{\mathbf R}^k \rightarrow  {\mathbf R}$ being a continuous, strictly positive symmetric function, and such that there exists a sequence of continuous functions $f_1, f_2, \ldots, f_k, \ldots$ each defined on ${\mathbf R}$ so that, for every $k$,
\[
{\mathfrak f_k}(\x)= \frac{f_1\wedge f_2 \wedge \ldots \wedge f_k(\x)} {\Delta_k(\x)}
\]
 for $\x\in {\mathbf R}^k$ with all coordinates distinct. In view of equation \eqref{eq:REP},  for $\boldsymbol{\mathfrak f} \in {\mathfrak F}_n$ with bounded components, we can define an ${\mathfrak F}_n$-valued process $(\boldsymbol{\mathfrak f}_t)_{t \geq 0}$ with initial value $\boldsymbol{\mathfrak f}$ by 
 \begin{equation}
\boldsymbol{\mathfrak f}_t= \bigl(   M_1^{{\mathfrak f}_1}(t,\cdot), M_2^{{\mathfrak f}_2}(t,\cdot), \ldots, M_n^{{\mathfrak f}_n}(t,\cdot) \bigr)
\end{equation}
We will denote, for any  $t>0$,  the ${\mathfrak F}_n$-valued random variable  $\boldsymbol{\mathfrak f}_t$ by ${\mathcal M}_t \boldsymbol{\mathfrak f}$.  For $0<s<t$, replacing the white noise in \eqref{eq:MnSPDEg} by its shift $\dot{W}^{(s)}(\cdot,\cdot)= \dot{W}(s+\cdot,\cdot)$ we can define analogously ${\mathcal M}_{s,t} \boldsymbol{\mathfrak f}$ where now $\boldsymbol{\mathfrak f}$ might be a random element of $\boldsymbol{\mathfrak F}_n$ determined by the white noise prior to time $s$, and whose $p$th-moments are uniformly bounded. Then, for bounded ${\mathfrak f}$, the existence and uniqueness results proved in \cite{LW15} imply that
\begin{equation}
\label{eq:Mevolv}
{\mathcal M}_t \boldsymbol{\mathfrak f}  = {\mathcal M}_{s,t}  {\mathcal M}_s \boldsymbol{\mathfrak f}.
\end{equation} 

Recall the definition of ${\mathfrak L}_n$ from Theorem \ref{thm:Markov}.
We are able to define a bijective correspondence between ${\mathfrak F}_n$ and ${\mathfrak L}_n$. Setting 
\[
\ell_k( x) = c_n^{-1} {\mathfrak f}_k( x {\mathbf 1}_k)
\]
defines a map ${\mathcal R}: ({\mathfrak f}_1, {\mathfrak f}_2, \ldots, {\mathfrak f}_n) \mapsto (\ell_1, \ell_2, \ldots, \ell_n )$. Somewhat suprisingly this is invertible, with an inverse which takes an explicit form - see Propositions \ref{representation1} and \ref{representation2}. In fact it's easy to see from the form of ${\mathcal R}^{-1}$ that if ${\mathfrak F}_n$ and ${\mathfrak L}_n$ are each equipped with the topology of locally uniform convergence, then ${\mathcal R}$ beceomes a homeomorphism between these spaces.

We use $ {\mathcal R}$ to induce a map
\[
\langle \cdot, \cdot \rangle_{{\mathfrak L}_n}: {\mathfrak L}_n \times {\mathfrak L}_n\rightarrow (0,\infty]^n,
\]
This is defined by setting $\langle {\boldsymbol\ell},   {\boldsymbol\ell}^\prime \rangle_{{\mathfrak L}_n}= (\beta_1, \beta_2, \ldots, \beta_n)$,
where
\[ 
  \beta_k= \frac{1}{k!}\int_{{\mathbf R}^k}   {\mathfrak f}_k(\x) {\mathfrak f}^\prime_k(\x) \Delta_k(\x)^2 \;\rd{\x},
\]
with ${\boldsymbol\ell}={\mathcal R}({\boldsymbol{\mathfrak f}})$ and ${\boldsymbol\ell}^\prime={\mathcal R}({\boldsymbol{\mathfrak f}}^\prime)$.

\begin{theorem}
\label{mainthm2}
For each $x\in{\mathbf R}$, the multiline line process $({\boldsymbol\tau}_{t})_{t>0}=({\boldsymbol\tau}(t,x,\cdot))_{t>0}$ is an ${\mathfrak L}_n$-valued process whose evolution is given by
\[
{\boldsymbol\tau}_{t}={\mathcal R}   {\mathcal M}_{s,t} {\mathcal R}^{-1} {\boldsymbol\tau}_{s} \quad \text { for $0<s<t$}.
\]
Moreover, we have the following flow property, for any $0<s<t$, and for any $x,y\in{\mathbf R}$, with probability one,
\[
{\boldsymbol\tau}(t,x,y)= \langle {\boldsymbol\tau}(s, x,\cdot),  {\boldsymbol\tau}(s,t,\cdot,y) \rangle_{{\mathfrak L}_n},
\]
where ${\boldsymbol\tau}(s,t, \cdot, \cdot)$ is defined analogously to ${\boldsymbol\tau}(t-s, \cdot, \cdot)$ but replacing the white noise with the shifted noise $\dot{W}^{(s)}$.
\end{theorem}

In the light of this theorem, it is clear how the multiline process should be defined starting from an initial condition ${\boldsymbol\ell }\in {\mathfrak L}_n$.  Provided ${\mathcal R}^{-1}({\boldsymbol\ell}) $ is bounded, then  
\begin{equation}
{\boldsymbol\tau}_t^{\boldsymbol \ell}= {\mathcal R} {\mathcal M}_t  {\mathcal R}^{-1}({\boldsymbol  \ell})
\end{equation}
defines an ${\mathfrak L}_n$-valued process, which starts from  ${\boldsymbol\ell }$.  We also have the representation  
\begin{equation}
{\boldsymbol\tau}_t^{\boldsymbol \ell} (y)= \langle {\boldsymbol \ell} , {\boldsymbol\tau}_t(\cdot,y) \rangle_{{\mathfrak L}_n}.
\end{equation}
In essence there is a conjugacy between the random dynamical system on ${\mathfrak L}_n$ which drives the evolution of the multilayer process, and the random dynamical system on ${\mathfrak F}_n$ induced by the stochastic heat equation. More precisely, if ${\mathfrak F}^0_n\subset {\mathfrak F}_n$ is invariant for (some extension of) ${\mathcal M}_{s,t}$ then we can define a conjugate dynamical system on ${\mathcal R}\bigl( {\mathfrak F}^0_n \bigr)\subset {\mathfrak L}_n$ via 
\[
{\mathcal T}_{s,t} = {\mathcal R}  {\mathcal M}_{s,t}  {\mathcal R}^{-1}.
\]

\section{Wronskians for non-smooth functions and the Toda equations}
\label{sec:wronskians}

Suppose $f_1, f_2, \ldots, f_n:{\mathbf R}\rightarrow {\mathbf R} $ are continuous functions such that
\[
\frac{f_1 \wedge f_2 \wedge \ldots \wedge f_n (\x)}{\Delta_n(\x)}
\]
extends continuously to ${\mathbf R}^n$. Then we define the Wronskian $W(f_1, f_2, \ldots f_n) (x)$ to be  the product of the constant $c_n^{-1}$ and the value of the extension at $\x=x{\bf 1}_n$. If the functions $f_1, f_2,\ldots, f_n$ are smooth enough, then this agrees with the usual definition. Hirota's bilinear derivative, which we  denote as $D(f_1,f_2)$ is  the Wronskian  $W(f_2,f_1)=-W(f_1,f_2)$. 
Suppose that $f_2$ is strictly positive. Then we can write 
\[
\frac{f_1 \wedge f_2 (x_1,x_2)}{(x_2-x_1)}= \frac{ f_2(x_2) f_2(x_1)}{(x_2-x_1)} \left( \frac{f_1(x_1)}{f_2(x_1)} -\frac{ f_1(x_2)}{f_2(x_2)}\right),
\]
and so we see a continuous extension to $\{x_1=x_2\}$ existing is equivalent to  the ratio $f_1/f_2$ being continuously differentiable, and then we have
\begin{equation}
\label{ratio}
\frac{f_1(b)}{f_2(b)}-\frac{f_1(a)}{f_2(a)}=\int_a^b \frac{D(f_1,f_2)(x)}{f_2(x)^2} \;\rd{x}.
\end{equation}
The following key lemma is the basis for all the results of this paper. It extends to our generalised Wronskians an easy but important consequence of the Jacobi identity for classical Wronskians.

\begin{lemma}
\label{keylemma}
Suppose that $f_1 \wedge \ldots \wedge f_n/\Delta_n$, $f_1 \wedge  \ldots \wedge  f_n \wedge g/\Delta_{n+1}$,   $f_1 \wedge \ldots \wedge  f_n \wedge h/\Delta_{n+1} $ and $f_1 \wedge  \ldots \wedge f_n \wedge g \wedge h/\Delta_{n+2}$ extend to continuous functions on ${\mathbf R}^{n}$, ${\mathbf R}^{n+1}$, ${\mathbf R}^{n+1}$ and  ${\mathbf R}^{n+2}$ respectively. Further assume $f_1 \wedge  \ldots \wedge  f_n \wedge g/\Delta_{n+1}$ extends as a strictly positive function. Then
\[
\frac{W(f_1, \ldots, f_n,h)}{  W(f_1, \ldots, f_n,g) } \text{ is differentiable}
\]
and
\[
D( W(f_1, \ldots, f_n,h),  W(f_1, \ldots, f_n,g)) = W(f_1, \ldots, f_n)W(f_1, \ldots, f_n,g,h).
\]
\end{lemma}
\begin{proof}   
Denote ${f_1 \wedge f_2 \ldots \wedge f_n }/{\Delta_n}$,${f_1 \wedge f_2 \ldots \wedge f_n \wedge g}/{\Delta_{n+1}}$, ${f_1 \wedge f_2 \ldots \wedge f_n \wedge h }/{\Delta_{n+1}}$ and ${f_1 \wedge f_2 \ldots \wedge f_n \wedge g \wedge h}/{\Delta_{n+2}}$ by $F$, $G$,$H$ and $K$ respectively. Let ${\bf x}= (x_1, x_2, \ldots, x_n)$ be a point in $ {\mathbf R}^n$, and for $ y,z \in {\mathbf R}$ consider points $({\bf x},y),({\bf x},y)\in{\mathbf R}^{n+1}$ and  $({\bf x},y,z)\in{\mathbf R}^{n+2}$. The Jacobi identity gives
\begin{multline*}
f_1 \wedge f_2 \wedge \ldots \wedge f_n \wedge g( {\bf x},y) f_1 \wedge f_2 \wedge \ldots \wedge f_n  \wedge h ({\bf x},z)-\\
f_1 \wedge f_2 \wedge \ldots \wedge f_n \wedge h( {\bf x},y) f_1 \wedge f_2 \wedge \ldots \wedge f_n \wedge g  ({\bf x},z)=
 \\
f_1 \wedge f_2 \wedge \ldots \wedge f_n( {\bf x}) f_1 \wedge f_2 \wedge \ldots \wedge f_n \wedge g \wedge h ({\bf x},y,z)
\end{multline*}
Now, as can be easily be checked  directly,
\[
\frac{\Delta_{n+1}({\bf x},y) \Delta_{n+1}({\bf x},z)}{\Delta_{n}({\bf x}) \Delta_{n+2}({\bf x},y,z)}=\frac{1}{z-y},
\]
and so we deduce that
\[
\frac{1}{z-y} \bigl(  G( {\bf x},y) H ({\bf x},z)-H( {\bf x},y)G  ({\bf x},z)\bigr)=F({\bf x}) K({\bf x},y,z).
\]
Notice that the continuity of $F$, $G$, $H$ and $K$ implies that we no longer need to assume that $x_1, x_2, \ldots x_n,y,z$ are all distinct. Moreover dividing through by $ G( {\bf x},y) G( {\bf x},z)$ which by hypothesis is strictly positive, the same continuity implies that we can take a limit as $y-z$ tends to $0$ and thereby deduce that 
\[
y \mapsto \frac{H({\bf x},y)}{G({\bf x},y)} \text{ is continuously differentiable} 
\]
and that its derivative is ${F({\bf x}) K({\bf x},y,y)}/{G({\bf x},y)^2}.$  But $G$ and $H$, and hence the ratio $H/G$ are symmetric functions on ${\mathbf R}^{n+1}$, and so in fact the partial derivative in each variable of  $H/G$ exists and is continuous. By a standard result from calculus this implies that $H/G$ is a differentiable function on ${\mathbf R}^{n+1}$. In particular, we have
\[
 x \mapsto \frac{H(x {\bf 1}_n)}{G(x {\bf 1}_n)} \text{ is continuously differentiable}   
\]
and its derivative, given by summing the partial derivatives, is
\[
(n+1)\frac{F({x\bf 1_n}) K(x{\bf 1}_{n+2})}{G(x{\bf 1}_{n+1})^2}.
\]
In view of the definition of the Wronskian $W(f_1, \ldots f_n)= c_n^{-1} F({x\bf 1_n})$ and similarly for the other Wronskians appearing in the statement, the lemma is proved, noting that$ c_n c_{n+2}/c_n^2=1/(n+1)$.
\end{proof}

We turn now to determinants constructed from a kernel function of two real variables. When a continuous  extension exists in this setting we can define a quantity that generalizes the notion of bi-directional Wronskians, and as in the case of smooth functions these give rise to solutions of the two dimensional Toda equations. This is the content of the following proposition, the proof of which is to first apply the preceeding lemma in one of the variables of the kernel, and then repeat the arguments used in the proof of the lemma in the second variable. 

\begin{proposition}
Suppose that $u$ is a continuous function defined on ${\mathbf R} \times {\mathbf R}$ such that, for every $n \geq 1$,
\[
  ( \x,\y) \mapsto \frac{ \det[ u(x_i,y_j) ]_{n \times n}}{ \Delta_n(\x) \Delta_n(\y)} 
\]
extends to a strictly postive, continuous function $M_n$ on ${\mathbf R}^n\times {\mathbf R}^n$. Define, for $n\geq 1$, $\tau_n$ via
\[
\tau_n(x,y)=c_n^{-2} M_n( x{\bf 1}_n, y{\bf 1}_n),
\]
and $\tau_0(x,y)=1$.
Then $(\tau_n; n \geq 1)$ satisfy the two dimensional Toda equations in integrated form
\[
  \log \left( \frac{\tau_n(a,c)\tau_n(b,d)}{\tau_n(a,d)\tau_n(b,c)}\right)= \int_a^b \int_c^d \frac{\tau_{n-1}(x,y) \tau_{n+1}(x,y)}{\tau_n(x,y)^2} \;\rd{y}\rd{x}.
\]
\end{proposition}
\begin{proof}
For any choice of distinct $x_1,x_2,\ldots ,x_n, y,z\in{\mathbf R}$ we can apply the preceeding lemma to the functions $f_i(\cdot)=u(x_i, \cdot)$ for $i=1,2,\ldots,n$, $g(\cdot)= u(y, \cdot) $, and $h(\cdot)= u(z, \cdot) $ the necessary extensions existing by virtue of the existence of $M_n,M_{n+1}$ and $M_{n+2}$. Noting that
\[
M_n( \x, c{\bf 1}_n)= c_n^{-1} \frac{ W(f_1,f_2, \ldots, f_n)(c)} {\Delta_n(\x)},
\]
with similar expressions for the Wronskians $ W(f_1,f_2, \ldots, f_n,g)$, $ W(f_1,f_2, \ldots, f_n,h)$ and  $W(f_1,f_2, \ldots, f_n,g,h)$ and using  
\[
\frac{\Delta_{n+1}({\bf x},y) \Delta_{n+1}({\bf x},z)}{\Delta_{n}({\bf x}) \Delta_{n+2}({\bf x},y,z)}=\frac{1}{z-y},
\]
we can write the conclusion of applying the lemma as
\begin{multline*}
\frac{1}{z-y}\left( \frac{M_{n+1}( (\x,z), d{\bf 1}_{n+1})} {M_{n+1}( (\x,z), c{\bf 1}_{n+1})}- \frac{M_{n+1}( (\x,y), d{\bf 1}_{n+1})} {M_{n+1}( (\x,y), c{\bf 1}_{n+1})}  \right)=\\  \frac{M_{n+1}( (\x,y), d{\bf 1}_{n+1})} {M_{n+1}( (\x,z), c{\bf 1}_{n+1})}
\int_c^d \frac{ M_n(\x,v{\bf 1}_n) M_{n+2}((\x,y,z),v{\bf 1}_{n+2})}{M_{n+1}((\x,y),v{\bf 1}_{n+1})^2} \rd{v}. 
\end{multline*}
Since the right hand side is a continuous function of $\x$, $y$ and $z$ we can both drop the assumption that $x_1, x_2, \ldots, x_n,y$ are all distinct, and then let $z-y$ tend to $0$ and deduce that 
\[
y \mapsto \frac{M_{n+1}( (\x,y), d{\bf 1}_{n+1})} {M_{n+1}( (\x,y), c{\bf 1}_{n+1})}
\]
is continuously differentiable with derivative
\[
\frac{M_{n+1}( (\x,y), d{\bf 1}_{n+1})} {M_{n+1}( (\x,y), c{\bf 1}_{n+1})}
\int_c^d \frac{ M_n(\x,v{\bf 1}_n) M_{n+2}((\x,y,y),v{\bf 1}_{n+2})}{M_{n+1}((\x,y),v{\bf 1}_{n+1})^2} \;\rd{v}.
\]
 $M_{n+1}( \cdot, d{\bf 1}_{n+1})/ M_{n+1}( \cdot, c{\bf 1}_{n+1})$ being a symmetric function on ${\mathbf R}^{n+1}$, this shows that all its partial deriavtives exist and are continuous, and hence, that it is in fact continuously differentiable. Then we calculate 
\begin{multline*}
  \frac{\rd{}}{\rd{x}}\left(\frac{M_{n+1}( x{\bf 1}_{n+1}, d{\bf 1}_{n+1})}{ M_{n+1}( x{\mathbf 1}_{n+1}, c{\bf 1}_{n+1})}\right)= \\
(n+1)\frac{M_{n+1}( (x{\bf 1}_{n+1}, d{\bf 1}_{n+1})} {M_{n+1}( ((x{\bf 1}_{n+1}, c{\bf 1}_{n+1})}
\int_c^d \frac{ M_n( (x{\bf 1}_{n},v{\bf 1}_n) M_{n+2}((x{\bf 1}_{n+2},v{\bf 1}_{n+2})}{M_{n+1}((x{\bf 1}_{n+1},v{\bf 1}_{n+1})^2} \;\rd{v}
\end{multline*}
from which the Toda equation follows easily.

\end{proof}

In view of Theorem \ref{thm:MnRegularity}, and the representation \eqref{eq:MnKM}, the above proposition applies directly to $u(x,y)=u(t,x,y)$,  the jointly continuous version of the solution to the stochastic heat equation.

\section{Construction of the conjugacy and proof of theorems}
\label{sec:conjugacy}

We begin with   a lemma that ensures all the Wronskians we need exist.

\begin{lemma}
\label{minors}
Suppose $f_1, f_2, \ldots, f_n:{\mathbf R}\rightarrow {\mathbf R} $ are continuous functions such that $f_1 \wedge f_2 \wedge \ldots \wedge f_n /\Delta_n$ extends continuously to ${\mathbf R}^n$. Suppose also that $f_1,f_2, \ldots, f_n$  are linearly independent. Then, $f_{i_1} \wedge f_{i_2} \wedge \ldots \wedge f_{i_k} /\Delta_k$ extends continuously to ${\mathbf R}^k $ for any subset $\{i_1,i_2, \ldots, i_k\} \subseteq \{1,2,\ldots,n\}$. 
\end{lemma}
\begin{proof}
According to Laplace's expansion for determinants,
\begin{multline*}
f_1 \wedge f_2 \wedge \ldots \wedge f_n (x_1,x_2,\ldots,x_n)= \\
 \sum_{k=1}^n (-1)^{k+1} f_k(x_1) f_1\wedge \ldots \wedge f_{k-1}\wedge f_{k+1} \wedge \ldots \wedge f_n ( x_2,\ldots,x_n).
\end{multline*}
Now the linear independence of $f_1, f_2, \ldots, f_n$ ensures $f_1 \wedge f_2 \wedge \ldots \wedge f_n$ is not identically zero, and so we can find $z_1, z_2, \ldots, z_n$ so that the matrix $A$ with $(j,k)$th entry $f_{k}(z_j)$ is invertible. So the $n$ linear equations arising from Laplace's expansion 
by choosing $ x_1=z_1, z_2, \ldots,z_n$ successfully can be solved to give, for each $k$, 
\begin{multline*}
f_1\wedge \ldots \wedge f_{k-1}\wedge f_{k+1} \wedge \ldots \wedge f_n ( x_2,\ldots,x_n)= \\ (-1)^{k+1} \sum_{j=1}^n A^{-1}_{kj} f_1 \wedge f_2 \wedge \ldots \wedge f_n (z_j,x_2,\ldots,x_n).
\end{multline*}
Now dividing this equation through by $\Delta_{n-1}(x_2,\ldots x_n)$ the right hand side is a linear combination of continuous functions on ${\mathbf R}^{n-1}$ which proves the assertion of the lemma for $\{i_1,i_2, \ldots, i_{n-1}\} = \{1,2,\ldots,n\} \setminus \{k\}$. The general case follows by iterating the argument.
\end{proof}

A sequence of continuous functions $f_1, f_2, \ldots, f_n$ with the property that for each $k \in \{1,2,\ldots, n\}$, $f_1\wedge f_2\wedge \ldots \wedge f_k(\x)/\Delta_k(\x)$ extends to a continuous, strictly positive function on ${\mathbf R}^k$ form an extended, complete Tchebysheff system, as defined in \cite[Chapter 6]{Ka68}, where $f_1,\ldots, f_n$ are assumed to be $(n-1)$-times continuously differentiable so that Wronskians exist in the classical sense. The following proposition is essentially a generalisation to the non-differentiable case of a representation for such a system in terms of its successive Wronskians, see \cite[Chapter 6, Theorem 1.1]{Ka68}. In the statement of the proposition the  integrals denote any choice of indefnite integral.

\begin{proposition}
\label{representation1} Suppose  $f_1, f_2, \ldots, f_n$  are continuous functions with the property that 
that, for each $k \in \{1,2,\ldots, n\}$, $f_1\wedge f_2\wedge \ldots \wedge f_k(\x)/\Delta_k(\x)$ extends to a continuous, strictly positive function on ${\mathbf R}^k$.   Denote $W(f_1, f_2, \ldots, f_k)$ by $\ell_k$ for $ 1\leq k\leq n$.
Then $f_1 \wedge f_2 \wedge \ldots \wedge f_k$  for $k=1, 2, \ldots, n$ can be recovered from $\ell_1, \ell_2, \ldots, \ell_n$: indeed if we define
\begin{align*}
\tilde{f}_1&= \ell_1, \\
\tilde{f}_2&= \ell_1\int \frac{\ell_2}{\ell_1^2}, \\
\tilde{f}_3&= \ell_1\int \frac{\ell_2}{\ell_1^2} \int \frac{\ell_1 \ell_3}{\ell^2_2}, \\
 \tilde{f}_n&= \ell_1\int \frac{\ell_2}{\ell_1^2} \int \frac{\ell_1 \ell_3}{\ell^2_2} \int \cdots \int \frac{\ell_{n-2}\ell_n}{\ell_{n-1}^2}, \\
\end{align*}
then 
\[
f_1 \wedge f_2 \wedge \ldots \wedge f_k= \tilde{f}_1 \wedge \tilde{f}_2 \wedge \ldots \wedge \tilde{f}_k.
\]
\end{proposition}
\begin{proof}
First note that since $W(f_1,f_k)$ exists by the previous lemma, and $f_1$ is positive,
\[
f_k= f_1 \int \frac{W(f_1,f_k)}{f_1^2},
\]
as in equation \eqref{ratio}, with some unknown constant of integration.
Then since for $2 \leq r \leq k$, all the necessary continuous extensions exist by the preceeding lemma, and $W(f_1, f_2, \ldots ,f_{r})$ being strictly positive by hypothesis we can apply Lemma \ref{keylemma} to obtain
\begin{multline*}
W( f_1, f_2,\ldots, f_{r-1},f_k)= W( f_1, f_2,\ldots, f_r) \int  \frac{ W( f_1, \ldots, f_{r-1}) W( f_1, \ldots, f_{r-1} f_{r},f_k)  } {W( f_1, f_2,\ldots, f_r)^2}.
\end{multline*}
Applying this successively  for $r=2,  \ldots, k$, and substituting into our expression for $f_k$ from the previous step of the iteration, gives the result, noting the  unknown constants of integration  don't affect the value of the determinants.
\end{proof}

Recalling from Section \ref{sec:dynamics} the definition of the map ${\mathcal R}$ the preceeding proposition implies that ${\mathcal R}: {\mathfrak F}_n \rightarrow {\mathfrak L}_n$ is injective and gives an explicit form for its left inverse. To show that ${\mathcal R}$ is surjective,  we need to consider whether the extension of  $\tilde{f}_1 \wedge \tilde{f}_2 \wedge \ldots \wedge \tilde{f}_k/ \Delta_k$ exists when $\tilde{f}_1, \ldots, \tilde{f}_k$ are  constructed as in the proposition but $\ell_1, \ldots, \ell_k$ are arbitrary strictly positive continuous functions. This follows easily from the following proposition.

\begin{proposition}
\label{representation2}
Suppose $\rho_1,\rho_2, \ldots, \rho_n$ are continuous  functions on ${\mathbf R}$,  and that 
$ g_1=\rho_1, g_2= \rho_1 \int \rho_2,  \ldots, g_n= \rho_1 \int \rho_{2} \int \cdots \int \rho_n.$ Then 
$(g_1\wedge g_2 \wedge \ldots \wedge g_n)/ \Delta_n$ has a continuous extension to ${\mathbf R}^{n}$ and 
\[
W(g_1,g_2, \ldots, g_n)= \rho_1^n W(h_1, h_2, \ldots, h_{n-1}),
\]
where 
$ h_1=\rho_2, h_2= \rho_2 \int \rho_3,  \ldots h_{n-1}= \rho_2 \int \rho_{3} \int \cdots \int \rho_n.$ 
If $\rho_1,\rho_2, \ldots, \rho_n$  are all strictly positive then so too is the extension of $(g_1\wedge g_2\wedge \ldots \wedge g_n)/ \Delta_n$.
\end{proposition}
\begin{proof}
This is by induction on $n$. Let $H$ denote the continuous extension of  $(h_1\wedge h_2 \wedge \ldots \wedge h_{n-1})/ \Delta_{n-1}$ which we assume exists by the inductive hypothesis.
Denote $(g_1\wedge g_2 \wedge \ldots \wedge g_n)/ \Delta_n$ by $G$. Since this is symmetric, it  is enough to show that this extends from a function on  $W_n= \{\x\in {\mathbf R}^n: x_1<x_2 <\cdots <x_n\}$ to a continuous function on $ \{\x\in {\mathbf R}^n: x_1\leq x_2 \leq \cdots \leq x_n\}$.  Elementary operations on determinants verify that
\[
  g_1\wedge g_2 \wedge \ldots \wedge g_n (\x)=\left\{\prod_{k=1}^n g_1(x_k)  \right\} \int_{\y \preceq \x} h_1\wedge h_2 \wedge \ldots \wedge h_{n-1}(\y) \;\rd{\y},
\]
where $\y \preceq  \x$ means $\y$ is  interlaced with $\x$, that is,
\[
x_1\leq y_1 \leq x_2 \leq y_2 \leq \cdots \leq x_{n-1} \leq y_{n-1} \leq  x_n.
\]
Consequently, 
\[
  G(\x)= \int H(\y)  \frac{\Delta_{n-1} (\y)}{\Delta_n(\x) }{\mathbf 1}( \y \preceq \x) \;\rd{\y}.
\]
Now the kernel $(n-1)!\frac{\Delta_{n-1} (\y)}{\Delta_n(\x) }{\mathbf 1}( \y \preceq \x) $ is well known to have the following random matrix interpretation. If ${\mathcal  M}$ is an $n\times n$ random Hermitian matrix having the GUE distribution, and $\Pi {\mathcal  M}$ denotes its principal $(n-1) \times (n-1)$ minor then the conditional distribution of the eigenvalues of  $\Pi {\mathcal  M}$ given that the eigenvalues of  ${\mathcal  M}$ equal $\x \in W_n$ has density  given by this kernel. Consequently we can write 
\[
G(\x) = \frac{1}{(n-1)!}{\mathbf E} \bigl[  H( \Pi  {\mathcal  M} ) | \text{sp}({\mathcal M})=\x \bigr]=  \frac{1}{(n-1)!} {\mathbf E} \bigl[  H( \Pi ( {\mathcal  U}  {\mathcal D}(\x)  {\mathcal U}^*))\bigr]
\]
where  we interpret $H$ as a continuous function of $(n-1) \times (n-1)$ Hermitian matrices that is invariant under the action of the unitary group, ${\mathcal  U}$ denotes an $n\times n$ random unitary  matrix distributed according to Haar measure, and $ {\mathcal D}(\x)$ is a diagonal matrix with diagonal entries $(x_1,x_2, \ldots, x_n)$. But  this representation naturally defines $G$ on the boundary of the Weyl chamber $W_n$, and moreover since $\Pi$ and $H$ are continuous, the dominated convergence theorem implies that $G$ is continuous too. Moreover, if $H$ and $g_1$ are both strictly positive then so too must be $G$.

Finally, the relationship between Wronskians is simply the fact that 
\[
G(x{\mathbf 1}_n)=\frac{c_n}{c_{n-1}} g_1(x)^n H(x{\mathbf 1}_{n-1}).
\]
\end{proof}

To prove Theorem \ref{mainthm2} we must verify that the ``fundamental solution'' ${\boldsymbol\tau}(t,x, y)$ evolves according to  
\[
{\boldsymbol\tau}_{t}={\mathcal R}   {\mathcal M}_{s,t} {\mathcal R}^{-1} {\boldsymbol\tau}_{s},
\]
and satisfies the flow property stated in the theorem. 

Take a sequence of points $\x(k) \in {\mathbf R}^n_{\neq}$ with $\x(k) \rightarrow x{\mathbf 1}_n$ as $k \rightarrow \infty$. Then by virtue of Theorem \ref{thm:MnRegularity}, with probability one, 
\[
M_n( t, \x(k), \y) \rightarrow M_n(t, x{\mathbf 1}_n, \y) \text{ locally uniformly in $\y \in {\mathbf R}^n$}.
\]
Consequently,
\[
c_n^{-1}M_n( t, \x(k), y{\mathbf 1}_n) \rightarrow \tau_n(t,x,y) \text{ locally uniformly in $y \in {\mathbf R}$.}
\]
Now $M_n( t, \x(k), \cdot)  \in C(\mathbf{R}^n, \mathbf{R})$ for each $k$, and   ${\mathcal R}$ being a homeomorphism, we deduce that $(M_k( t, x{\mathbf 1}_k, \cdot))_{1\leq k\leq 
n}  \in {\mathfrak F}_n$ and 
\begin{equation}
\label{secondv}
{\mathcal R} \bigl(M_k(t, x{\mathbf 1}_k, \cdot), 1\leq k\leq n \bigr)=  {\boldsymbol\tau}(t, x, \cdot).
\end{equation}
Now it was proved in \cite{LW15} that,  for a fixed $s>0$ and  $k\geq 1$, $M_k(s, x{\mathbf 1}_k,\y)$ has uniformly bounded $p$th moments as $y\in {\mathbf R}^k$ varies, and consequently  we can  solve the evolution equation  for $M_k$ with shifted white noise and initial data  $\bigl(M_k(s, x{\mathbf 1}_k,\y); \y \in{\mathbf R}\bigr)$.  Thus, as at equation as a \eqref{eq:Mevolv}, it is meaningful to apply  ${\mathcal M}_{s,t}$ to    $(M_k( s, x{\mathbf 1}_k, \cdot))_{1\leq k\leq 
n}$,  giving  $(M_k( t, x{\mathbf 1}_k, \cdot))_{1\leq k\leq n}$. It now  follows from \eqref{secondv}, with the time $t$ replaced by the time $s$, that  ${\boldsymbol\tau}_t(x, \cdot)$  satisfies
\[
{\boldsymbol\tau}_{t}={\mathcal R}   {\mathcal M}_{s,t} {\mathcal R}^{-1} {\boldsymbol\tau}_{s}. 
\]

Corollary 6.2 of \cite{OW16},  states that, for given $\x, \y \in {\mathbf R}^n$ and $0<s<t$, with probability one,
\begin{equation}
\label{mflow}
M_n(t,\x,\y)= \frac{1}{n!} \int_{\mathbf{R}^n} M_n(s,\x, \z)M_n(s,t,\z,\y) \Delta_n(\z)^2 \;\rd{\z}
\end{equation}
where  $M_n(s,t, \cdot,\cdot)$ is defined  as $M_n(t-s,\cdot,\cdot)$ by the chaos expansion but with the shifted white noise $\dot{W}^{(t-s)}(\cdot,\cdot)= \dot{W}((t-s)+\cdot,\cdot)$. 
Notice that the argument that gave \eqref{secondv}  is easily adapted to give 
\begin{equation}
\label{firstv}
{\mathcal R} \bigl(M_k(t, \cdot, y{\mathbf 1}_k ), 1\leq k\leq n \bigr)=  {\boldsymbol\tau}_t( \cdot,y).
\end{equation}
 And now combining \eqref{secondv}, \eqref{firstv}, and the definition of the pairing $\langle \cdot,\cdot\rangle_{{\mathfrak L}_n}$ when we take $\x=x{\mathbf 1}_n$ and $\y=y{\mathbf 1}_n$,  we obtain the desired equation for ${\boldsymbol\tau}(t,x, y)$ from \eqref{mflow}.

Theorem \ref{thm:Markov} is a fairly immediate consequence of the construction in Theorem \ref{mainthm2}. 
Let $0\leq s < t$ and fix $x\in\R$.
Let $C\in\mathscr{B}({\mathfrak L }_n)$ be a Borel set. Denote $\mb{Z}_t = \big(Z_1(t,x,\cdot),\ldots,Z_n(t,x,\cdot)\big)$. Let $\bigl(\mathscr{F}_t\bigr)_{t \geq 0}$ denote the filtration generated by the white noise $\big( \dot{W}(s,A); 0\leq s\leq t, A\in  \mathscr{B}({\mathbf R}) \bigr)$. We will show that the conditional probability that $\mb{Z}_t\in C$ given $\mathscr{F}_s$ is measurable with respect to $\sigma\big(\mb{Z}_s\big)$. 

By \eqref{eq:MnBoundary}, the process $\mb{Z}_t$ is proportional to $\big(M_1(t,x\mb{1}_1,\cdot\mb{1}_1),\ldots,M_n(t,x\mb{1}_n,\cdot\mb{1}_n)\big)$ and since $M_k(t,x\mb{1}_k,\cdot) $ can almost surely be identified as the solution of the evolution equation \eqref{eq:MnSPDE} with initial condition $M_k(s,x\mb{1}_k,\cdot) $ and driven by the shifted white noise which is independent of ${\mathscr F}_s$, we have,
\[
 \mathbb{P}[\mb{Z}_t\in C | \mathscr{F}_s]  \text{ is measurable with respect to $\big(M_1(s,x\mb{1}_1,\cdot),\ldots,M_n(s,x\mb{1}_n,\cdot)\big)$.}
\]
However, for each $1 \leq k \leq n$, $M_k(s,x\mb{1}_k,\cdot)$ is a function of $\mb{Z}_s$ and the result follows.

\section{Appendix}\label{sec:appendix}

We show that 
\[
  M_n(t,\x,\y) = \frac{\det[u(t,x_i,y_j)]_{i,j=1}^n}{\Delta_n(\x)\Delta_n(\y)}
\]
solves
\[
  M_n(t,\x,\y) = \frac{p_n^*(t,\x,\y)}{\Delta_n(\x)\Delta_n(\y)} + A_n \int_0^t \int_{\R^n} Q_{t-s}(\y,\z) M_n(s,\x,\z) \;\rd{\z_*} \W{s}{z_1},
\]
where $Q_t(\y,\z) = \frac{\Delta_n(\z)}{\Delta_n(\y)}p_n^*(t,\y,\z)$, $\rd{\z_*} = \rd{z_2}\ldots\rd{z_n}$ and $A_n = \frac{1}{(n-1)!}$.
It suffices to prove that $K_n(t,\x,\y) := \det[u(t,x_i,y_j)]_{i,j=1}^n$ satisfies
\begin{equation}
  K_n(t,\x,\y) = p_n^*(t,\x,\y) + A_n \int_0^t \int_{\R^n} p_n^*(t-s,\y,\z) K_n(s,\x,\z) \;\rd{\z_*} \W{s}{z_1},
  \label{eq:KnMild}
\end{equation}
then the result follows upon dividing by $\Delta_n(\x)\Delta_n(\y)$.
First note that $u$ is the solution to the following mild equation
\begin{align*}
  u(t,x,y) 
    &= p_t(x-y) + \int_0^t \int_\R p_{t-s}(y-z) u(s,x,z) \;\W{s}{z} \notag \\
    &= p_t(x-y) + I(t,x,y).
  \label{eq:SHEmild}
\end{align*}
Using this and expanding the determinant we have
\begin{align*}
  K_n(t,\x,\y) 
    &= \sum_{\sigma} (-1)^\sigma \prod_{i=1}^n \big( p_t(x_{\sigma i}-y_i) + I(t,x_{\sigma i},y_i) \big) \\
    &= \sum_{m=0}^n \sum_\sigma \sum_{(\mb{i},\mb{j})\in D_m} (-1)^\sigma \prod_{r=1}^m I(t,x_{\sigma i_r},y_{i_r}) \prod_{r=m+1}^n p_t(x_{\sigma j_r}-y_{j_r}),
\end{align*}
where $D_m := \{(\mb{i}, \mb{j}) : 1\leq i_1<\cdots<i_m\leq n, 1\leq j_{m+1}<\cdots<j_n\leq n, i_k \neq j_l \;\forall k,l \}$ and empty products are defined to be equal to 1.

For martingales $X^1,\ldots, X^n$ we have the following product formula
\begin{align*}
  \prod_{i=1}^n X_t^i = \prod_{i=1}^n X_0^i + \sum_{i=1}^n \int_0^t \prod_{j\neq i} X_s^j \;\rd{X_s^i} + \sum_{i\neq j} \int_0^t \prod_{k\neq i,j} X_s^k \;\rd{\langle X^i,X^j\rangle_s},
\end{align*}
which follows from the familiar formula for the $n=2$ case and induction.
Using this we have
\begin{align*}
  K_n(t,\x,\y) 
  &= \sum_{m=0}^n \sum_\sigma \sum_{(\mb{i},\mb{j})\in D_m} (-1)^\sigma \prod_{r=m+1}^n p_t(x_{\sigma j_r}-y_{j_r}) \\
  &\qquad\times \bigg( \sum_{k=1}^m \int_0^t \prod_{r\neq k} I_s^{\sigma i_r} \;\rd{I_s^{\sigma i_k}} - \sum_{k\neq l} \int_0^t \prod_{r\neq k,l} I_s^{\sigma i_r} \;\rd{\langle I^{\sigma i_k},I^{\sigma i_l}\rangle_s} \bigg),
\end{align*}
where $I_t^{\sigma i} := I(s,x_{\sigma i},y_i)$ and by \cite[Theorem 5.26]{Kh09}
\begin{equation*}
  \langle I^{\sigma i}, I^{\sigma j} \rangle_t = \int_0^t \int_\R p_{t-s}(y_i-z)p_{t-s}(y_j-z)u(s,x_{\sigma i},z)u(s,x_{\sigma j},z) \;\rd{z}\rd{s}.
  \label{eq:quadraticVariation}
\end{equation*}
Consider $\sigma^\p = \sigma\circ(i,j)$ then $\langle I^{\sigma^\p i}, I^{\sigma^\p j} \rangle_t = \langle I^{\sigma i}, I^{\sigma j} \rangle_t$.
Moreover since $(-1)^{\sigma^\p} = -(-1)^\sigma$ we have, by considering such pairs of permutations, that for each $m = 2,\ldots,n$
\[
  \sum_{\sigma} \sum_{(\mb{i},\mb{j})\in D_m} \sum_{k\neq l} (-1)^\sigma \int_0^t \prod_{r\neq k,l} I_s^{\sigma i_r} \;\rd{\langle I^{\sigma i_k},I^{\sigma i_l}\rangle_s} = 0.
\]
Therefore,
\begin{equation}
  K_n(t,\x,\y) = \sum_{m=0}^n \sum_\sigma \sum_{(\mb{i},\mb{j})\in D_m} \sum_{k=1}^m (-1)^\sigma \prod_{r=m+1}^n p_t(x_{\sigma j_r}-y_{j_r}) \int_0^t \prod_{r\neq k} I_s^{\sigma i_r} \;\rd{I_s^{\sigma i_k}}.
  \label{eq:LHS}
\end{equation}

On the other hand,
\begin{align}
  \int_0^t \int_{\R^n} & p_n^*(t-s,\y,\z) K_n(s,\x,\z) \;\rd{\z_*} \W{s}{z_1} \notag \\ 
  &= \sum_\sigma\sum_\pi (-1)^\sigma (-1)^\pi \int_0^t \int_{\R^n} \prod_{i=1}^n p_{t-s}(y_{\pi i}-z_i) u(s,x_{\sigma i}, z_i) \;\rd{\z_*}\W{s}{z_1}.
  \label{eq:RHS1}
\end{align}
Observe that
\begin{align*}
  u(t,x,y) 
  &= p_t(x-y) + I(t,x,y) \\
  &= \int_\R p_{t-s}(y-z) u(s,x,z) \;\rd{z} + \int_s^t \int_\R p_{t-r}(y-z) u(r,x,z) \;\W{r}{z},
\end{align*}
and so
\[
  \int_\R p_{t-s}(y-z) u(s,x,z) \;\rd{z} = p_t(x-y) + I(s,x,y).
\]
Using this the right hand side of \eqref{eq:RHS1} is equal to
\begin{align*}
  \sum_\sigma\sum_\pi & (-1)^\sigma (-1)^\pi \int_0^t \int_\R p_{t-s}(y_{\pi 1}-z) u(s,x_{\sigma 1},z) \\
  &\qquad\qquad\times \prod_{i=2}^n \big( p_t(x_{\sigma i}-y_{\pi i}) + I(s,x_{\sigma i},y_{\pi i}) \big) \;\W{s}{z} \\
  &= \sum_{m=1}^n \sum_\sigma\sum_\pi \sum_{(\mb{i},\mb{j})\in D_m^\p} (-1)^\sigma (-1)^\pi \prod_{r=m+1}^n p_t(x_{\sigma j_r}-y_{\pi j_r}) \\
  &\qquad\qquad\times \int_0^t \int_\R p_{t-s}(y_{\pi 1}-z) u(s,x_{\sigma 1},z) \prod_{r=2}^m I(s,x_{\sigma i_r},y_{\pi i_r}) \;\W{s}{z},
\end{align*}
where $D_m^\p = \{ (\mb{i},\mb{j}) : 2\leq i_2<\cdots<i_m\leq n, 2\leq j_{m+1}<\cdots<j_n\leq n, i_k\neq j_l \;\forall k,l\}$.
Note that there are $\binom{n-1}{m-1}$ terms in the sum over $D_m^\p$.

Let $(\mb{a},\mb{b})\in D_m$ where $D_m$ was defined above then for $1\leq k\leq m$, split the sum over $\pi\in S_n$ into groups $G_k(\mb{a},\mb{b})$ of permutations such that $\pi 1 = a_k$, $\pi i_r \in \{a_i : 1\leq i \leq m, i\neq k\}$ and for $r = 2,\ldots,m$ and $\pi i_r \in \{b_{m+1},\ldots,b_n\}$ for $r = m+1,\ldots,n$.
Then the right hand side of the previous display is equal to
\begin{align*}
  &\sum_{m=1}^n \sum_{(\mb{a},\mb{b})\in D_m} \sum_{k=1}^m \sum_{\pi\in G_k(\mb{a},\mb{b})} \sum_{(\mb{i},\mb{j})\in D_m^\p} \sum_\sigma (-1)^\sigma (-1)^\pi \prod_{r=m+1}^n p_t(x_{\sigma j_r}-y_{\pi j_r}) \\
  &\qquad\qquad\times \int_0^t \int_\R p_{t-s}(y_{\pi 1}-z) u(s,x_{\sigma 1},z) \prod_{r=2}^m I(s,x_{\sigma i_r},y_{\pi i_r}) \;\W{s}{z} \\
  &= \sum_{m=1}^n \sum_{(\mb{a},\mb{b})\in D_m} \sum_{k=1}^m \sum_\sigma A_n^{-1} (-1)^\sigma \prod_{r=m+1}^n p_t(x_{\sigma b_r}-y_{b_r}) \\
  &\qquad\qquad\times \int_0^t \int_\R p_{t-s}(y_{a_k}-z) u(s,x_{\sigma a_k},z) \prod_{r\neq k} I(s,x_{\sigma a_r},y_{a_r}) \;\W{s}{z},
\end{align*}
where the last equality is due to the fact that for each $\mb{a}$, $\mb{b}$ and $k$ each term in the sum $\sum_{\pi\in G_k(\mb{a},\mb{b})} \sum_{(\mb{i},\mb{j})\in D_m}$ are equal and there are in total $\binom{n-1}{m-1}(m-1)!(n-m)! = A_n^{-1}$ terms in the sum.
Finally, observing that the $m=0$ term in \eqref{eq:LHS} is equal to $p_n^*(t,\x,\y)$ shows that $\det[u(t,x_i,y_j)]_{i,j=1}^n$ satisfies equation \eqref{eq:KnMild}.

In the calculations above we saw stochastic integrals of the form
\begin{equation}
  \int_0^t \int_\R p_{t-s}(y_1-z) u(s,x_1,z) \prod_{i=2}^m I(s,x_i,y_i) \;\W{s}{z}.
  \label{eq:stochasticIntegral}
\end{equation}
This integral is well defined since the integrand is adapted and continuous and by H\"older's inequality
\begin{align}
  \int_0^t \int_\R & p_{t-s}(y_1-z)^2 \bigg\V u(s,x_1,z) \prod_{i=2}^m I(s,x_i,y_i) \bigg\V_2^2 \;\rd{z}\rd{s} \notag \\
  &\leq \int_0^t \int_\R p_{t-s}(y_1-z)^2 \V u(s,x_1,z) \V_{2m}^2 \prod_{i=2}^m \V I(s,x_i,y_i) \V_{2m}^2 \;\rd{z}\rd{s},
  \label{eq:L2Estimate}
\end{align}
where $\V\cdot\V_p := (\E[|\cdot|^p])^{1/p}$.
By the Burkholder--Davis--Gundy inequality, there is a constant $c = c(m)$ such that
\begin{align*}
  \V I(s,x,y) \V_{2m}^2 \leq c \int_0^s \int_\R p_{t-r}(y-z)^2 \V u(r,x,z)\V_{2m}^2 \;\rd{z}\rd{r}.
\end{align*}
Since $\V u(r,x,z) \V_p \leq c p_r(x-z)$ for a constant $c$ depending on $p$, see for example \cite[equation 3.19]{BC95}, we have for constants depending only on $m$
\begin{align*}
  \V I(s,x,y) \V_{2m}^2 
  &\leq c^\p \int_0^s \int_\R p_{t-r}(y-z)^2 p_r(x-y)^2 \;\rd{z}\rd{r} \\
  &\leq c^{\p\p} p_{t/2}(x-y) \int_0^s \frac{1}{\sqrt{t-r}} \frac{1}{\sqrt{r}} \;\rd{r} \\
  &\leq c^{\p\p} \pi p_{t/2}(x-y),
\end{align*}
where in the last line we used the fact that
\[
  \int_0^s \frac{1}{\sqrt{t-r}} \frac{1}{\sqrt{r}} \;\rd{r} \leq \int_0^t \frac{1}{\sqrt{t-r}} \frac{1}{\sqrt{r}} \;\rd{r} = \pi.
\]
Therefore, the right hand side of \eqref{eq:L2Estimate} is less than
\begin{align*}
  c^{\p\p\p} \prod_{i=2}^m p_{t/2}(x_i-y_i) \int_0^t \int_\R p_{t-s}(y_1-z)^2 p_s(x_1-z)^2 \;\rd{z}{s} \leq c^{\p\p\p\p} \prod_{i=1}^m p_{t/2}(x_i-y_i) < \infty.
\end{align*}
Consequently, by \cite[Proposition 3.1]{CD13}, the integral \eqref{eq:stochasticIntegral} is well defined.

\bibliographystyle{halpha}

\bibliography{./Biblio.bib}          

\begin{thebibliography}{DLLD16}

\bibitem[BBO09]{BBO09}
Philippe Biane, Philippe Bougerol, and Neil O'Connell.
\newblock Continuous crystal and {D}uistermaat-{H}eckman measure for {C}oxeter
  groups.
\newblock {\em Adv. Math.}, 221(5):1522--1583, 2009.

\bibitem[BC95]{BC95}
Lorenzo Bertini and Nicoletta Cancrini.
\newblock The stochastic heat equation: {F}eynman--{K}ac formula and
  intermittence.
\newblock {\em J. Statist. Phys.}, 78(5-6):1377--1401, 1995.

\bibitem[BDJ99a]{BDJ99}
Jinho Baik, Percy Deift, and Kurt Johansson.
\newblock {On the distribution of the length of the longest increasing
  subsequence of random permutations}.
\newblock {\em J. Amer. Math. Soc.}, 12(4):1119--1178, 1999.

\bibitem[BDJ99b]{BDJ99b}
Jinho Baik, Percy Deift, and Kurt Johansson.
\newblock {On the distribution of the length of the second row of a Young
  diagram under Plancherel measure}.
\newblock {\em Geom. Funct. Anal}, 10:25, 1999.

\bibitem[BOO00]{BOO00}
Alexei Borodin, Andrei Okounkov, and Grigori Olshanski.
\newblock Asymptotics of plancherel measures for symmetric groups.
\newblock {\em J. Amer. Math. Soc.}, 13(3):481--515, 2000.

\bibitem[CD15]{CD13}
Le~Chen and Robert~C. Dalang.
\newblock Moments and growth indices for the nonlinear stochastic heat equation
  with rough initial conditions.
\newblock {\em Ann. Probab.}, 43(6):3006--3051, 2015.

\bibitem[CH14]{CH14}
Ivan Corwin and Alan Hammond.
\newblock {Brownian Gibbs property for Airy line ensembles}.
\newblock {\em Inventiones mathematicae}, 195(2):441--508, 2014.

\bibitem[CH15]{CH13}
Ivan Corwin and Alan Hammond.
\newblock {KPZ} line ensemble.
\newblock {\em Probability Theory and Related Fields}, pages 1--119, 2015.

\bibitem[CN16]{CN16}
Ivan Corwin and Mihai Nica.
\newblock {Intermediate disorder directed polymers and the multi-layer
  extension of the stochastic heat equation}, 2016, arXiv:1603.08168.

\bibitem[COSZ14]{COSZ14}
Ivan Corwin, Neil O'Connell, Timo Sepp\"{a}l\"{a}inen, and Nikos Zygouras.
\newblock {Tropical combinatorics and {W}hittaker functions}.
\newblock {\em Duke Math. J.}, 163(3):513--563, 2014.

\bibitem[DLLD15]{DL15}
Andrea De~Luca and Pierre Le~Doussal.
\newblock Crossing probability for directed polymers in random media.
\newblock {\em Phys. Rev. E}, 92:040102, Oct 2015.

\bibitem[DLLD16]{DL16}
Andrea De~Luca and Pierre Le~Doussal.
\newblock Crossing probability for directed polymers in random media. ii. exact
  tail of the distribution.
\newblock {\em Phys. Rev. E}, 93:032118, Mar 2016.

\bibitem[Dys62]{Dy62}
Freeman~J. Dyson.
\newblock A {B}rownian-motion model for the eigenvalues of a random matrix.
\newblock {\em J. Math. Phys.}, 3:1191--1198, 1962.

\bibitem[Hai13]{Ha13}
Martin Hairer.
\newblock {Solving the KPZ equation}.
\newblock {\em Ann. Math.}, 178(2):559--664, 2013.

\bibitem[Hir04]{Hir04}
Ryōgo Hirota.
\newblock {\em The direct method in soliton theory}.
\newblock Cambridge University Press, Cambridge, U.K., 2004.

\bibitem[Joh00]{Joh99}
Kurt Johansson.
\newblock {Shape Fluctuations and Random Matrices}.
\newblock {\em Comm. Math. Phys.}, 209(2):51, 2000.

\bibitem[Joh01]{Joh01}
Kurt Johansson.
\newblock Discrete orthogonal polynomial ensembles and the plancherel measure.
\newblock {\em Ann. Math.}, 153(1):259--296, 2001.

\bibitem[Joh03]{Joh03}
Kurt Johansson.
\newblock {Discrete Polynuclear Growth and Determinantal Processes}.
\newblock {\em Commun. Math. Phys.}, 242(1-2):277--329, 2003.

\bibitem[Kar68]{Ka68}
S.~Karlin.
\newblock {\em Total Positivity}.
\newblock Number v. 1 in Total Positivity. Stanford University Press, 1968.

\bibitem[Kho09]{Kh09}
Davar Khoshnevisan.
\newblock A primer on stochastic partial differential equations.
\newblock In {\em A Minicourse on Stochastic Partial Differential Equations},
  volume 1962 of {\em Lecture Notes in Mathematics}, pages 1--38. Springer
  Berlin Heidelberg, 2009.

\bibitem[KM59]{KM59}
Samuel Karlin and James McGregor.
\newblock Coincidence probabilities.
\newblock {\em Pacific J. Math.}, 9:1141--1164, 1959.

\bibitem[LW15]{LW15}
Chin~Hang Lun and Jon Warren.
\newblock Continuity and strict positivity of the multi-layer extension of the
  stochastic heat equation.
\newblock 2015, {arXiv:1506.09030}.

\bibitem[OW16]{OW16}
Neil O'Connell and Jon Warren.
\newblock A {M}ulti-{L}ayer {E}xtension of the {S}tochastic {H}eat {E}quation.
\newblock {\em Comm. Math. Phys.}, 341(1):1--33, 2016.

\bibitem[PS02]{PS02}
Michael Pr\"ahofer and Herbert Spohn.
\newblock Scale invariance of the {PNG} droplet and the {A}iry process.
\newblock {\em J. Stat. Phys.}, 108:1071--1106, 2002.

\end{thebibliography}

\end{document}